\documentclass[12pt]{amsart}

\usepackage{latexsym,amssymb,amsfonts,amsmath,mathrsfs,float,hyperref}
\usepackage{xcolor}
\hypersetup{
    colorlinks,
    linkcolor={black},
    citecolor={black},
}

  \renewcommand{\Pr}{\mbox{\rm Pr}}
  \newcommand{\Exp}{{\mathbb{E}}}

  \newcommand{\C}{\mathbb{C}} % complex numbers
  \newcommand{\Z}{\mathbb{Z}} % integers
  \newcommand{\F}{\mathbb{F}} % field
  \DeclareMathOperator{\im}{im} % image
  \DeclareMathOperator{\vspan}{Span} % kernel

  \newcommand{\st}{:\,} % "such that" to define sets

  \newcommand{\eps}{\varepsilon}

  \newcommand{\ceil}[1]{\lceil{#1}\rceil}
  \newcommand{\floor}[1]{\lfloor{#1}\rfloor}

  \DeclareMathOperator{\symten}{Sym}
  \DeclareMathOperator{\arank}{arank}
  \DeclareMathOperator{\prank}{prank} 
  \DeclareMathOperator{\bias}{bias} 
   
  \newcommand{\beq}{\begin{equation}}
  \newcommand{\eeq}{\end{equation}}
  \newcommand{\beqn}{\begin{equation*}}
  \newcommand{\eeqn}{\end{equation*}}
  \newcommand{\beqr}{\begin{eqnarray}}
  \newcommand{\eeqr}{\end{eqnarray}}
  \newcommand{\beqrn}{\begin{eqnarray*}}
  \newcommand{\eeqrn}{\end{eqnarray*}}
  \newcommand{\bmline}{\begin{multline}}
  \newcommand{\emline}{\end{multline}}
  \newcommand{\bmlinen}{\begin{multline*}}
  \newcommand{\emlinen}{\end{multline*}}

  \theoremstyle{plain}
  \newtheorem{theorem}{Theorem}[section]
  \newtheorem{lemma}[theorem]{Lemma}
  \newtheorem{proposition}[theorem]{Proposition}
  
  \newtheorem{corollary}[theorem]{Corollary}
  
  \theoremstyle{definition}
  \newtheorem{definition}[theorem]{Definition}

  \theoremstyle{remark}

  \renewenvironment{proof}[1][]{
    	\begin{trivlist}
     	\item[\hspace{\labelsep}{\em\noindent Proof#1:\/}]}
     	{{\hfill$\Box$}
    	\end{trivlist}
  }
  
  \makeatletter
  \newtheorem*{rep@theorem}{\rep@title}
  \newcommand{\newreptheorem}[2]{%
  \newenvironment{rep#1}[1]{%
  \def\rep@title{#2 \ref{##1}}%
  \begin{rep@theorem}}%
  {\end{rep@theorem}}}
  \makeatother

  \newreptheorem{lemma}{Lemma}
  \newreptheorem{theorem}{Theorem}
  \newreptheorem{corollary}{Corollary}
  \newreptheorem{proposition}{Proposition}
  \newreptheorem{conjecture}{Conjecture}

\begin{document}
\title{Subspaces of tensors with high analytic rank}
\author{Jop Bri\"{e}t}
\address{CWI, Science Park 123, 1098 XG Amsterdam, The Netherlands}
\email{j.briet@cwi.nl}

\thanks{The author  is  supported  by   the  Gravitation grant  NETWORKS-024.002.003  from  the Dutch Research Council (NWO)}
\maketitle

\begin{abstract}
It is shown that for any subspace $V\subseteq \F_p^{n\times\cdots\times n}$ of $d$-tensors,
if $\dim(V) \geq tn^{d-1}$, then
there is  subspace $W\subseteq V$ of dimension at least
$t/(dr) - 1$
whose nonzero elements all have analytic rank~$\Omega_{d,p}(r)$.
As an application, we generalize a result of Altman on Szemer\'edi's theorem with random differences.
\end{abstract}

\section{Introduction}

In~\cite{Meshulam:1985}, Meshulam proved the following result.

\begin{theorem}[Meshulam]\label{thm:meshulam}
Let~$\F$ be a field and let~$V\subseteq M_n(\F)$ be a subspace of~$n\times n$ matrices.
If $\dim(V) > rn$, then~$V$ contains a matrix of rank at least~$r+1$.
\end{theorem}

Here we prove a version of this result for tensors over finite fields.
Identify a $d$-linear form~$T:\F^n\times\cdots\times\F^n\to \F$ with the order-$d$ tensor with $(i_1,\dots,i_d)$-coordinate $T(e_{i_1},\dots,e_{i_d})$, where~$e_i$ is the $i$th standard basis vector in~$\F^n$.
A tensor of order~$d$ will be referred to as a $d$-tensor.
The notion of rank for tensors we consider is the analytic rank, introduced by Gowers and Wolf in~\cite{GowersWolf:2011}.
 
\begin{definition}[Bias and analytic rank]
Let~$d\geq 2$ and $n\geq 1$ be integers.
Let~$\F$ be a finite field and let~$\chi:\F\to\C$ be a nontrivial additive character.
Let~$T\in \F^{n\times\cdots\times n}$ be a~$d$-tensor.
Then, the \emph{bias} of~$T$ is defined by\footnote{Here and elsewhere, for a finite set~$X$, we denote $\Exp_{x_1,\dots,x_k\in X}f(x_1,\dots,x_k) = |X|^{-k}\sum_{x_1\in X}\dots \sum_{x_k\in X}f(x_1,\dots,x_k)$ and $\Pr_{x_1,\dots,x_k\in X}$ denotes the probability with respect to independent uniformly distributed elements~$x_1,\dots,x_k\in X$.}
\beqn
\bias(T)
=
\Exp_{x_1,\dots,x_d\in \F^n}\chi\big(T(x_1,\dots,x_d)\big),
\eeqn
and the \emph{analytic rank} of~$T$ is defined by
\beqn
\arank(T) = -\log_{|\F|}\bias(T).
\eeqn
\end{definition}

The bias is well-defined, since its value is independent of the choice of nontrivial additive character and it is not hard to see that it is real and nonnegative.
Moreover, for any $d \geq 2$, the analytic rank is at most~$n$ and for matrices ($d=2$), the analytic rank is the ordinary matrix rank.\\

Our version of Theorem~\ref{thm:meshulam} is then as follows. 

\begin{theorem}\label{thm:main}
For every finite field~$\F$ and integer~$d\geq 2$, there is a $c\in (0,1]$ such that the following holds.
Let $n\geq t\geq r\geq 1$ be integers and $V\subseteq \F^{n\times \cdots \times n}$ be a subspace of~$d$-tensors.
If $\dim(V) \geq tn^{d-1}$, then there is a subspace~$W\subseteq V$ of dimension at least~$\frac{t}{dr}-1$ such that every nonzero element in~$W$ has analytic rank at least~$cr$. 
\end{theorem}

Theorem~\ref{thm:main} gives an analogue of Theorem~\ref{thm:meshulam} asserting that if~$V$ has dimension at least~$rn^{d-1}$, then it contains a tensor of analytic rank at least~$\Omega_{\F,d}(r)$.
The same statement holds for another notion of tensor rank, namely the partition rank, which originated in~\cite{naslund:2017}.

\begin{definition}[Partition rank]
Let $d \geq 2$ and $n\geq 1$ be integers.
A $d$-linear form $T:\F^n\times \cdots\times \F^n\to \F$ has partition rank~1 if there exist integers $1\leq e,f\leq d-1$ such that $e+f = d$, a partition $\{i_1,\dots,i_{e}\},\{j_1,\dots,j_{f}\}$ of~$[d]$ and $e$- and $f$-linear forms~$T_1,T_2$ (respectively)  such that for any $x_1,\dots,x_d\in \F^n$,
\beqn
T(x_1,\dots,x_d)
=
T_1\big(x_{i_1},\dots,x_{i_{e}}\big)\,
T_2\big(x_{j_1},\dots,x_{j_{f}}\big).
\eeqn
The partition rank of~$T$ is the smallest $r$ such that $T = T_1 + \cdots + T_r$, where each~$T_i$ has partition rank~1.
\end{definition}

Partition rank is always at most~$n$ and for matrices is also equal to the usual rank.
Independently, Kazhdan and Ziegler~\cite{Kazhdan:2018} and Lovett~\cite{Lovett:2019} proved that $\prank(T) \geq \arank(T)$ and
so Theorem~\ref{thm:main} holds for the partition rank as well.
This implies that the parameters of Theorem~\ref{thm:main} are close to optimal.
Indeed, if~$U\subseteq \F^n$ is a~$t$-dimensional subspace and~$V = \F^{n\times\cdots \times n}$ is the set of $(d-1)$-tensors, then~$U\otimes V$ is a $(tn^{d-1})$-dimensional subspace of $d$-tensors containing only tensors of partition rank (and so analytic rank) at most~$t$.
In the other direction, partition and analytic rank are polynomially related.
Independently, Mili\'cevi\'c in~\cite{Milicevic:2019} and Janzer in~\cite{Janzer:2019} proved that 
\beq\label{eq:arank-prank}
\prank(T) \leq O_{|\F|,d}\big(\arank(T)^D\big)
\eeq
for some~$D\leq 2^{2^{d^{O(1)}}}$. 
We refer to these papers for further information on bounds for specific values of~$d$.
\medskip

\subsection{Szemer\'edi's theorem with random differences}
We apply Theorem~\ref{thm:main} to a probabilistic version of Szemer\'edi's theorem~\cite{Szemeredi:1975}. 
For~$\eps\in (0,1]$ and integer~$k\geq 3$, Szemer\'edi's theorem asserts that any set $A\subseteq \Z/N\Z$ of size at least~$\eps N$ contains a proper $k$-term arithmetic progression ($k$-AP), provided~$N$ is large enough in terms of~$\eps$ and~$k$.
The setup for the probabilistic version is as follows.
Given a finite abelian group~$G$ of order~$N$ and positive integer~$m$, let~$S\subseteq G$ be a random subset formed by sampling $m$ elements from~$G$ independently and uniformly at random.
A general open problem is to determine the smallest~$m$ such that with high probability over~$S$, any set $A\subseteq G$ of size at least~$\eps N$ contains a proper $k$-AP with common difference in~$S$.
For $k = 3$, it was shown by Christ~\cite{Christ:2011} and Frantzikinakis, Lesigne and Wierdl~\cite{FLW:2012} that $m \geq \omega(\sqrt{N}\log N)$ suffices and for $k \geq 3$, it was shown by Gopi and the author in~\cite{BG:2018} that $m\geq \omega(N^{1 - \frac{1}{\ceil{k/2}}}\log N)$ does; see also~\cite{BDG:2019}.
In~\cite{FLW16} the authors conjecture that in the group~$\Z/N\Z$, for all fixed $k \geq 3$, already $m \geq \omega(\log N)$ would do.
However, in \cite{Altman:2019} Altman showed that in the finite field case, where $G =\F_p^n$ with~$p$ an odd prime, the analogous conjecture is false for 3-APs and that $m \geq \Omega_p(n^2)$ is necessary (we refer to this paper for more information).
Using Theorem~\ref{thm:main}, we generalize Altman's result to arbitrarily long~APs.

\begin{theorem}\label{thm:krandsz}
For every integer~$k\geq 3$ and prime~$p\geq k$ there is a constant~$C$ such that the following holds.
If $S\subseteq \F_p^n$ is a set formed by selecting at most ${n+k-2\choose k-1} - C(\log_p n)^2n^{k-2}$ elements independently and uniformly at random, then with probability $1 - o(1)$ there is a set $A\subseteq \F_p^n$ of size $|A| \geq \Omega_{k,p}(p^n)$ that contains no proper $k$-term arithmetic progression with common difference in~$S$.
\end{theorem}

In particular, for $N = p^n$ at least $\Omega((\log_p N)^{k-1})$ elements must be sampled for Szemer\'edi's theorem with random differences and $k$-APs over~$\F_p^n$.
Showing (much) stronger lower bounds, possibly over other groups (including non-abelian groups), is of interest for coding theory~\cite{BDG:2019}.
In~\cite{Altman:2019}, the case $k = 3$ of Theorem~\ref{thm:krandsz} is proved without squaring the logarithmic factor.
The proof given there uses both the analytic and algebraic characterization of matrix rank and can be generalized using the relations between analytic and partition rank.
But the best-known relations~\eqref{eq:arank-prank} cause the exponent of the logarithmic factor to blow up substantially and currently require fairly intricate proofs. 
Theorem~\ref{thm:main} allows one to avoid the use of partition rank altogether gives a proof based more easily-established results.

\subsubsection*{Acknowledgements.}
I thank Farrokh Labib and Michael Walter for useful discussions.

\section{Proof of Theorem~\ref{thm:main}}

We use some results of Lovett~\cite{Lovett:2019} and corollaries thereof.
Let~$\F$ be a finite field.
For a $d$-tensor $T\in \F^{n\times \cdots\times n}$ and set $S\subseteq [n]:=\{1,\dots,n\}$, denote by~$T_{|S}$ the principal sub-tensor obtained by restricting~$T$ to $S\times\cdots\times S$.
It will be convenient to slightly extend the definitions of bias and analytic rank.
For finite sets~$S_1,\dots,S_d$, $d$-tensor $T\in \F^{S_1\times\cdots\times S_d}$ and non-trivial additive character~$\chi:\F\to \C$, define
\beqn
\bias(T) = \Exp_{(x_1,\dots,x_d)\in \F^{S_1}\times\cdots\times \F^{S_d}}\chi\big(T(x_1,\dots,x_d)\big)
\eeqn
and define $\arank(T)$ as before.

\begin{lemma}[Lovett]\label{lem:arank_sub}
Let~$T\in \F^{n\times\cdots\times n}$ be a $d$-tensor and $S\subseteq [n]$.
Then, 
\beqn
\arank(T) \geq \arank\!\big(T_{|S}\big).
\eeqn
\end{lemma}

\begin{corollary}\label{cor:arank_sub}
Let~$T\in \F^{n\times\cdots\times n}$ be a $d$-tensor,
let~$S_1,\dots,S_d\subseteq [n]$ be sets of equal size and let $T'\in\F^{S_1\times\cdots\times S_d}$ be the restriction of~$T$ to $S_1\times\cdots\times S_d$.
Then, 
\beqn
\arank(T) \geq \arank(T').
\eeqn
\end{corollary}

\begin{proof}
Let~$\pi_2,\dots,\pi_d:[n]\to[n]$ be permutations such that $\pi_i(S_i) = S_1$.
Let~$Q$ be the $d$-tensor obtained by permuting the $i$th leg of~$T$ according to~$\pi_i$.
Then, since analytic rank is invariant under such permutations, it follows from Lemma~\ref{lem:arank_sub} that
\beqn
\arank(T) = \arank(Q) \geq \arank\!\big(Q_{|S_1}\big) = \arank(T'),
\eeqn
where the second equality follows since~$Q_{|S_1}$ is a permutation of~$T'$.
\end{proof}

\begin{lemma}[Lovett]\label{lem:lovettlem}
Let~$\chi:\F\to\C$ be a nontrivial additive character.
Let~$T\in \F^{n\times\cdots\times n}$ be a $d$-tensor and let $\F^n = U\oplus V$ for two subspaces~$U,V$.
Then, for any $v_1,\dots,v_d\in V$, 
\beqn
\Big|
\Exp_{u_1,\dots,u_d\in U}\chi\big(T(u_1+v_1,\dots,u_d+v_d)\big)\Big|
\leq
\Exp_{u_1,\dots,u_d\in U}\chi\big(T(u_1,\dots,u_d)\big).
\eeqn
\end{lemma}

\begin{corollary}\label{cor:arank_plus}
Let~$T\in\F^{n\times \cdots\times n}$ be a $d$-tensor
and $E_n = e_n\otimes\cdots\otimes e_n$ be the $d$-tensor with a~1 at its last coordinate and zeros elsewhere.
Then, there exists a $\lambda\in \F$ such that
\beqn
\arank(T+\lambda E_n) \geq  \arank\!\big(T_{|[n-1]}\big) + c_{\F,d},
\eeqn
where
\beq\label{eq:cFd}
c_{\F,d} = -\log_{|\F|}\Big(1 - \Big(\frac{|\F| - 1}{|\F|}\Big)^d\Big).
\eeq
\end{corollary}

\begin{proof}
Let $U= \vspan(e_1,\dots,e_{n-1})$ and $V$ be the line spanned by~$e_n$.
We consider the average bias of the tensor $T + \lambda E_n$, where~$\lambda$ is uniformly distributed over~$\F$.
This average equals
\begin{align*}
\Exp_{\lambda\in \F} \Exp_{u_1,\dots,u_d\in U}\Exp_{v_1,\dots,v_n\in V}\chi\big((T+\lambda E_n)(u_1+v_1,\dots,u_d+v_d)\big).
\end{align*}
The character expression factors as
\beqn
\chi\big(T(u_1+v_1,\dots,u_d+v_d)\big)
\chi(\lambda E_n(u_1+v_1,\dots,u_d+v_d)\big).
\eeqn
Writing $v_i = a_ie_n$, then the second factor simplifies to $\chi(\lambda a_1\cdots a_d)$.
Hence, the average bias of~$T + \lambda E_n$ equals
\begin{align*}
\Exp_{a\in \F^d}
\Big(
\Exp_{u_1,\dots,u_n\in U}
\chi\big(T(u_1+a_1e_n,\dots,u_d+a_de_n)\big)
\Big)
\Big(\Exp_{\lambda}
\chi(\lambda a_1\cdots a_d)
\Big).
\end{align*}
The expectation over~$\lambda$ equals $1[a_1\cdots a_d =0]$.
Hence, by H\"{o}lder's inequality, the average bias is at most
\beqn
\max_{v_1,\dots,v_d\in V}
\Big|
\Exp_{u_1,\dots,u_n\in U}
\chi\big(T(u_1+v_1,\dots,u_d+v_d)\big)
\Big|
\Pr_{a_1,\dots,a_d\in \F}[a_1\cdots a_d = 0].
\eeqn
The result now follows from Lemma~\ref{lem:lovettlem}.
\end{proof}
\medskip

We now prove Theorem~\ref{thm:main} following similar lines as Meshulam's proof of Theorem~\ref{thm:meshulam}.
\medskip

\begin{proof}[ of Theorem~\ref{thm:main}]
Order~$[n]^d$ lexicographically.
For a $d$-tensor $T\in \F^{n\times\cdots\times n}$, let~$\rho(T)$ denote its first nonzero coordinate.
Let~$T_1,\dots,T_{\dim(V)}$ be a basis for~$V$.
By Gaussian elimination (viewing the~$T_i$ as vectors in~$\F^{n^d}$), we can assume that the coordinates~$\rho(T_i)$ are pairwise distinct.

Cover~$[n]^d$ by the ``diagonal matchings'' given by
\begin{align*}
\big\{(0,n_1,\dots,n_{d-1})+(i,\dots,i) &\st i \in[n-\max_{l\in[d-1]}n_l]\big\}\\
\big\{(n_1,0,\dots,n_{d-1})+(i,\dots,i) &\st i \in[n-\max_{l\in[d-1]}n_l]\big\}\\
&\vdots\\
\big\{(n_1,\dots,n_{d-1},0)+(i,\dots,i) &\st i \in[n-\max_{l\in[d-1]}n_l]\big\},
\end{align*}
for $n_1,\dots,n_{d-1}\in \{0,\dots,n-1\}$.
This is a cover since the coordinate $(i_1,\dots,i_d)$ with $j = \min\{i_1,\dots,i_d\}-1$ lies in the matching whose smallest element (with respect to the product order) is $(i_1 - j,\dots,i_j-j)$.
Since there are at most~$dn^{d-1}$ such matchings and $\dim(V) \geq tn^{d-1}$, one of these matchings contains at least~$t/d$ of the coordinates~$\rho(T_1),\dots,\rho(T_{\dim(V)})$.
Let $s = \floor{t/dr}$, so that $rs \leq t/d$.

Relabelling if necessary, we can assume that~$\rho(T_1),\dots,\rho(T_{rs})$ lie in the same matching and that they are listed increasingly  according to the product order, so that
\beq\label{eq:Tirho}
\rho(T_i) = (n_1,\dots,n_{d}) + (f(i),\dots,f(i))
\eeq
for some $n_1,\dots,n_d\in\{0,\dots,n-1\}$ and strictly increasing function $f:[rs]\to [n]$.
For each $i\in[rs]$, let $Q_i \in \F^{rs\times\cdots \times rs}$ be 
tensor given by
\beqn
Q_i(i_1,\dots,i_d)
=
T_i\big((n_1,\dots,n_{d}) + (f(i_1),\dots,f(i_d))\big).
\eeqn
Then,~$Q_i$ is a sub-tensor of~$T_i$ obtained from its restriction to the rectangle $(n_1,\dots,n_d) + (\im(f))^d$.
Moreover, it follows from~\eqref{eq:Tirho} that $\rho(Q_i) = (i,\dots,i)$, which in turn implies that the restriction~$(Q_i)_{|[i]}$ is nonzero only on coordinate $(i,\dots,i)$.

Partition $[rs]$ into $s$ consecutive intervals $I_1,\dots,I_s$ of length~$r$ each.
We claim that for each $j\in [s]$, there is an $R_j\in \vspan(Q_i \st i\in I_j)$  such $\arank\big((R_j)_{|I_j}\big)\geq c_{\F,d}r$, for $c_{\F,d}$ as in~\eqref{eq:cFd}.
We prove the claim for $j =1$.
To this end, we show by induction on $i \in [r]$ that  $\vspan(Q_1,\dots,Q_i)$ contains a tensor $R$ whose restriction~$R_{|[i]}$ to $[i]\times\cdots\times[i]$ has analytic rank at least 
$ic_{\F,d}$.
For $i = 1$, the claim follows since $Q_1(1,\dots,1) = a$ for some $a\in \F^*$ and the bias of the $1\times\cdots\times 1$ tensor~$a$ equals
\begin{align*}
\Exp_{x_1,\dots,x_d\in \F}\chi(ax_1\cdots x_d)
&=
\Pr_{x_2,\dots,x_d\in \F}[x_2\cdots x_n =0]\\
&=
1 - \Big(\frac{|\F| - 1}{|\F|}\Big)^{d-1}\\
&\leq
|\F|^{-c_{\F,d}}.
\end{align*}
Assume the claim for $i \in [r-1]$ and let $R\in  \vspan(Q_1,\dots,Q_i)$ be such that $\arank(R_{|[i]}) \geq ic_{\F,d}$.
Since the restriction of~$(Q_{i+1})_{|[i+1]}$ is nonzero only on coordinate $(i+1,\dots,i+1)$,  it is a nonzero multiple of~$E_{i+1}$.
Hence, by Corollary~\ref{cor:arank_plus}, there is a $\lambda \in \F$ such that 
\begin{align*}
\arank\big((R+\lambda Q_{i+1})_{|[i+1]}\big) 
 &\geq \arank(R_{|[i]}) + c_{\F,d}
 \geq (i+1)c_{\F,d},
\end{align*}
which proves the claim.
For $j>1$ the claim is proved similarly, using induction on $i\in[r]$ to show that $\vspan(Q_{jr+1},\dots,Q_{jr + i})$ contains a tensor~$R$ such that $\arank(R_{|\{jr+1,\dots,jr+i\}}) \geq ic_{\F,d}$.

For $j\in [s]$, let $T^*_j \in \vspan(T_i \st i\in I_j)$ be the tensor whose  restriction to $(n_1,\dots,n_d) + (\im(f))^d$ equals~$R_j$.
Let $W = \vspan(T_1^*,\dots,T_s^*)$.
We claim that the space~$W$ meets the criteria of Theorem~\ref{thm:main}.
Since the sets~$I_j$ are pairwise disjoint and the tensors $T_1,\dots,T_{\dim(V)}$ linearly independent, it follows that $\dim(W) \geq s \geq \frac{t}{dr}-1$.
Let $\lambda\in \F^s\setminus\{0\}$ and let $j\in [s]$ be its first nonzero coordinate.
It follows from Corollary~\ref{cor:arank_sub} and~Lemma~\ref{lem:arank_sub} that 
\begin{align*}
\arank(\lambda_1 T_1^* + \cdots + \lambda_sT_s^*) &\geq \arank(\lambda_1 R_1 + \cdots + \lambda_sR_s)\\
&\geq 
\arank\big((\lambda_1 R_1 + \cdots + \lambda_sR_s)_{|I_j}\big)\\
&=
\arank\big(\lambda_j(R_j)_{|I_j}\big)\\
&\geq
c_{\F,d}r,
\end{align*}
where in the third line we used that $\lambda_k = 0$ for all $k\in[j-1]$ and $(R_k)_{|I_j} = 0$ for all $k\in \{j+1,\dots,s\}$, which holds since the restriction of $Q_i$ to $I_j$ is the zero tensor for all $i>j$.
Hence, every nonzero element of~$W$ has analytic rank at least~$c_{\F,d}r$.
\end{proof}

\section{Proof of Theorem~\ref{thm:krandsz}}
\label{sec:sz}

For positive integer~$d$ and $x\in \F^n$, denote $\varphi_d(x) = x\otimes\cdots\otimes x$ ($d$~times).
Then, for any $d$-tensor $T\in \F^{n\times\cdots\times n}$, we have
$
T(x,\dots,x) = \langle T,\varphi_d(x)\rangle$,
where $\langle\cdot,\cdot\rangle$ is the standard inner product.
Theorem~\ref{thm:krandsz} follows from the following two lemmas, the first of which is proved in~\cite{Altman:2019} and the second of which we prove below.

\begin{lemma}[Altman]\label{lem:indep}
Let $k\geq 3$ be an integer and~$p\geq k$ be a prime number.
Let $S\subseteq \F_p^n$ be such that the set $\varphi_{k-1}(S)$ is linearly independent.
Then, there exists a nonzero $(k-1)$-tensor $T\in \F_p^{n\times\cdots\times n}$ such that the set $\{x\in \F_p^n \st \langle T,\varphi_{k-1}(x)\rangle = 0\}$
contains no $k$-term arithmetic progressions with common difference in~$S$.
\end{lemma}

\begin{lemma}\label{lem:main}
For every integer~$d\geq 2$ and prime ${p\geq d+1}$, there is a~$C\in (0,\infty)$ such that the following holds.
Let $m = {n+d-1\choose d}$ and let $s \leq m- C (\log_p m)^2n^{d-1}$ be an integer.
Let $x_1,\dots,x_s$ be independent and uniformly distributed random vectors from~$\F_p^n$.
Then, $\varphi_d(x_1),\dots,\varphi_d(x_s)$ are linearly independent with probability~$1-o(1)$.
\end{lemma}

Theorem~\ref{thm:krandsz} now follows from Lemma~\ref{lem:main} with~$d = k-1$ and the Chevalley--Warning theorem~\cite[Chapter~6]{Lidl:1983}, which implies that the set from Lemma~\ref{lem:indep} has size~$\Omega_{k,p}(p^n)$.

Lemma~\ref{lem:main} follows from the following proposition, which in turn follows from Theorem~\ref{thm:main}.
A $d$-tensor is symmetric if it is invariant under permutations of its legs.
Let $\symten_d^n(\F_p)$ be the ${n + d-1\choose d}$-dimensional subspace of symmetric $d$-tensors.
Note that if $p>d$, then~$\langle\cdot,\cdot\rangle$ is non-degenerate on~$\symten_d^n(\F_p)$ since if $T$ is an element of this space with a nonzero $(i_1,\dots,i_d)$-coordinate, then by symmetry of~$T$ and the fact that $d!\not\equiv 0\pmod{p}$, we have
\beqn
\Big\langle T, \sum_{\pi\in S_d} e_{i_{\pi(1)}}\otimes\cdots\otimes e_{i_{\pi(d)}}\Big\rangle
=
\sum_{\pi\in S_d}\langle T, e_{i_{\pi(1)}}\otimes\cdots\otimes e_{i_{\pi(d)}}\rangle=
d!T_{i_1,\dots,i_d}
\ne 0.
\eeqn

\begin{proposition}\label{prop:inU}
For every integer~$d\geq 2$ and prime $p\geq d+1$, there is a~$C\in (0,\infty)$ such that the following holds.
Let~$t> 0$
and $U\subseteq \symten_d^n(\F_p)$ be a subspace of co-dimension at least $C 4^dt^2n^{d-1}$.
Then,
\beqn
\Pr_{x\in \F_p^n}[\varphi_{d}(x)\in U] \leq \frac{2}{p^{2t}}.
\eeqn
\end{proposition}

\begin{proof}
Let $V = U^\perp\subseteq \symten_d^n(\F_p)$.
Then, since $\langle\cdot,\cdot\rangle$ is non-degenerate, $U = V^\perp$ and $\dim(V) \geq C4^dt^2n^{d-1}$.
Moreover, if~$C$ is large enough in terms of~$d$ and~$p$, then it follows from Theorem~\ref{thm:main} that there is a subspace~$W\subseteq V$ of dimension~$m\geq 2^dt$ such that each nonzero element of~$W$ has analytic rank at least~$r\geq 2^dt$.
Hence, for $\omega = e^{2\pi i/p}$, we have
\begin{align*}
\Pr_{x\in \F_p^n}[\varphi_{d}(x)\in U]
&=
\Pr_{x\in \F_p^n}[\varphi_{d}(x)\in V^\perp]\nonumber\\
&\leq
\Pr_{x\in \F_p^n}[\varphi_{d}(x)\in W^\perp]\nonumber\\
&=
\Exp_{x\in \F_p^n}\Exp_{T\in W}\omega^{\langle \varphi(x), T\rangle}\nonumber\\
&\leq
\Big(
\Exp_{T\in W}\bias(T)
\Big)^{\frac{1}{2^{d-1}}}\label{eq:bias}\\
&\leq
\Big(
\frac{1}{p^m} + \frac{p^m - 1}{p^m}\,\frac{1}{p^r}
\Big)^{\frac{1}{2^{d-1}}}\nonumber\\
&\leq
\frac{2}{p^{2t}},\nonumber
\end{align*}
where the third line follows from~\cite[Lemma~3.5]{Altman:2019} and the fourth line follows from~\cite[Lemma~3.2]{GowersWolf:2011} and Jensen's inequality.
\end{proof}

A similar inequality to the one stated in Proposition~\ref{prop:inU} was proved in~\cite{Bhrushundi:2018} over~$\F_2^n$.
There, the full space of tensors is considered and the random element is of the form $x_1\otimes\cdots\otimes x_d$, where the $x_i$ are independent and uniformly distributed.
Similar to~\cite[Lemma~3.4]{Altman:2019} we can now prove Lemma~\ref{lem:main} .

\begin{proof}[ of Lemma~\ref{lem:main}]
The probability that $\varphi_d(x_1),\dots,\varphi_d(x_s)$ are linearly independent is at least
\begin{multline*}
\Pr[x_1\ne 0]\prod_{i=2}^s\Pr\big[\varphi_d(x_i) \not\in \vspan\big(\varphi_d(x_1),\dots,\varphi_d(x_{i-1})\big)\big]\\
\geq
\Big(1 - \max_{U\subseteq \symten_d^n(\F_p)}\Pr_{x\in \F_p^n}[\varphi_d(x)\in U]\Big)^s,
\end{multline*}
where the maximum is taken over $(s-1)$-dimensional subspaces.
Setting $t = \log_p m$ and using that~$s \leq m$, Proposition~\ref{prop:inU} then shows that this bounded from below by $(1 - \frac{2}{m^2})^s \geq 1 - O(1/m) \geq 1 - o(1)$.
\end{proof}

\bibliographystyle{alphaabbrv}
\bibliography{random_szemeredi}

\end{document}